\documentclass[11pt]{amsart}
\usepackage{setspace}
\usepackage{graphicx}
\usepackage{geometry}
\usepackage{multirow}
\usepackage{commath}
\usepackage{mdframed}
\usepackage{amsmath, amsthm}
\usepackage{thmtools}
\usepackage{bbm}
\usepackage{amssymb}
\usepackage{array}
\newcolumntype{M}[1]{>{\centering\arraybackslash}m{#1}}
\usepackage{mathtools}

\usepackage{amsthm}
\usepackage{float}
\usepackage{caption}
\usepackage{color}
\usepackage{enumerate}
\usepackage{tikz}
\usepackage{hyperref}
\hypersetup{
   colorlinks = true, 
   linkcolor = blue, 
   filecolor = magenta,      
   urlcolor = black, 
}
\usepackage{url}
\expandafter\def\expandafter\UrlBreaks\expandafter{\UrlBreaks
  \do\a\do\b\do\c\do\d\do\e\do\f\do\g\do\h\do\i\do\j
  \do\k\do\l\do\m\do\n\do\o\do\p\do\q\do\r\do\s\do\t
  \do\u\do\v\do\w\do\x\do\y\do\z\do\A\do\B\do\C\do\D
  \do\E\do\F\do\G\do\H\do\I\do\J\do\K\do\L\do\M\do\N
  \do\O\do\P\do\Q\do\R\do\S\do\T\do\U\do\V\do\W\do\X
  \do\Y\do\Z}
\usepackage{setspace}
\usepackage{systeme}
\usepackage{tabularx}
\usepackage{cite}
\usepackage{etoolbox}
\usepackage[british]{babel}
\newtheorem{theorem}{Theorem}[section]
\newtheorem{lemma}[theorem]{Lemma}
\newtheorem{cor}[theorem]{Corollary}

\newtheorem{question}[theorem]{Question}

\newtheorem{fact}[theorem]{Fact}

\newtheorem{prop}[theorem]{Proposition}

\newtheorem*{claim*}{Claim}
\newtheorem*{theorem*}{Theorem}
\newtheorem*{prop*}{Proposition}
\newtheorem*{lemma*}{Lemma}
\newtheorem*{keyobservation*}{Key Observation}
\newtheorem*{conjecture*}{Conjecture}
\newtheorem*{mainthm*}{Main Theorem}
\numberwithin{equation}{section}
\theoremstyle{definition}
\newtheorem{notation}[theorem]{Notation}
\newtheorem*{notation*}{Notation}
\newtheorem{defn}[theorem]{Definition}

\newtheorem{remark}[theorem]{Remark}

\newcommand{\tp}[0]{\mathrm{tp}}
\newcommand{\R}[0]{\mathbb{R}}
\newcommand{\Q}[0]{\mathbb{Q}}
\newcommand{\N}[0]{\mathbb{N}}
\newcommand{\Z}[0]{\mathbb{Z}}
\newcommand{\M}[0]{\mathcal{M}}
\newcommand{\Th}[0]{\text{Th}}
\newcommand{\Sym}[0]{\text{Sym}}
\newcommand{\ar}[0]{\text{ar}}
\newcommand{\monster}[0]{U}

\newcommand{\globalext}[1]{{{#1}\upharpoonright^\monster}}
\def\indep{\mathrel{\raise0.2ex\hbox{\ooalign{\hidewidth$\vert$\hidewidth\cr\raise-0.9ex\hbox{$\smile$}}}}}
\counterwithout{equation}{section}
\author{Mervyn Tong}
\address{Department of Pure Mathematics and Mathematical Statistics, Centre for Mathematical Sciences, Wilberforce Road, Cambridge CB3 0WB, United Kingdom}
\email{hwmt3@cam.ac.uk}
\title{Higher-arity distality and forking triviality}
\date{\today}
\begin{document}
\maketitle
\begin{abstract}
Answering a question of Goode, we show that $k$-triviality collapses to (1-)triviality among simple theories. In particular, every stable theory with quantifier elimination in a relational language of bounded arity is trivial.

We use our collapse result, along with other facts about $k$-triviality and $k$-total triviality, to generate examples of (strongly) $k$-distal theories. The collapse result immediately implies that no stable theory can be strictly $k$-distal for some $k\geq 3$, partially answering a question of Walker. Moreover, all known examples of non-distal (strongly) $k$-distal theories are $k$-ary, rendering (strong) $k$-distality moot as a $(k+1)$-ary dividing line; we give four classes of examples that are not $k$-ary. We also show that just as distality is not preserved under taking reducts, neither is (strong) $k$-distality.
\end{abstract}
\section{Introduction}
Model-theoretic classification theory has classically been centred on binary dividing lines: properties characterised by tameness for binary (partitioned) formulas. This is the case for NIP, which says that the graph of every binary formula omits some finite bipartite graph as an induced subgraph; for stability, which says that some finite half-graph is omitted in the above; as well as for distality, which says that the graph of every binary formula admits a useful kind of cell decomposition.

Recently, a programme has begun to develop higher-arity versions of these dividing lines. Indeed, higher-arity versions of NIP \cite{shelah1, shelah2, cpt, nadja1, nadja2, nadja3, terry2, chernikovtowsner} and stability \cite{terry1, terry2, terry3, chernikovtowsnerstability} have been studied in much depth, and many instructive examples of theories satisfying these notions have been found. Higher-arity distality has received less attention. Walker \cite{walker, walkerthesis} introduced the $(k+1)$-ary dividing lines of $k$-distality and strong $k$-distality, with 1-distality and strong 1-distality both equivalent to distality. No further work on these notions had emerged in the literature until very recently in \cite{artemfrancis}.

In particular, few examples of (strongly) $k$-distal theories are known that illuminate the nature of this property as a $(k+1)$-ary dividing line. The main goal of this paper is to supply these (although, as we shall see, there are some surprising by-products). Arguably the most notable omission is that every known example of a non-distal (strongly) $k$-distal theory is $k$-ary: it admits quantifier elimination in a $k$-ary relational language. It is not hard to see that every $k$-ary theory is (strongly) $k$-distal, but such examples are not instructive. Indeed, as a $(k+1)$-ary dividing line, (strong) $k$-distality is characterised by tameness for $(k+1)$-ary formulas, but in a $k$-ary theory, every $(k+1)$-ary formula is degenerate, in the sense that it is a Boolean combination of $k$-ary formulas.

We provide four classes of examples of non-distal (strongly) $k$-distal theories that are not $k$-ary; in fact, we show that they can be arbitrarily far from being $k$-ary. Say that a theory is \textit{strictly} (strongly) $k$-distal if it is (strongly) $k$-distal and not (strongly) $(k-1)$-distal.
\begin{prop}[Section \ref{subsecarity}]
    Let $k\in\N^+$.
    \begin{enumerate}[(i)]
        \item The theory of Goode's $(k+1)$-sorted labelled free pseudoplane is strictly 2-distal, strictly strongly $(2^k+1)$-distal, and not $(2^{k+1}-1)$-ary (hence not $(2^k+1)$-ary).
        \item The theory of the universal homogeneous kay-graph is strictly (strongly) $k$-distal and not $k$-ary.
        \item The theory of the Johnson graph $\mathfrak{J}(k)$ is strictly (strongly) 2-distal and not $g(k)$-ary, for some increasing unbounded function $g:\N^+\to\N^+$ with $g(2)\geq 2$.
        \item The theory of a permutation structure of Cherlin and Lachlan is strictly (strongly) 2-distal and not $l$-ary for any $l\in\N^+$.
    \end{enumerate}
\end{prop}
\begin{remark}
    Although there are (superstable) strongly $k$-distal theories that are not $k$-ary, we will show in forthcoming work that every NIP strongly $k$-distal theory is `locally $k$-ary', in the sense that every $(k+1)$-ary formula is `locally determined' by a conjunction of $k$-ary formulas.
\end{remark}

We also show that just as the reduct of a distal theory need not be distal, the same holds for (strongly) $k$-distal theories.
\begin{prop}[Proposition \ref{reducts}]
    For all $k\in\N^+$, $k$-distality and strong $k$-distality are not necessarily preserved under taking reducts.
\end{prop}

We note in passing that many examples of (strongly) $k$-distal theories, although not $k$-ary, are \textit{reducts} of $k$-ary theories, and we are unable to ascertain the precise relationship between these two properties --- see Section \ref{subsecreducts}.

Perhaps surprisingly, our proof methods rely heavily on a connection with forking triviality, as studied, for instance, in \cite{trivialpursuit,goode}. Say that a simple theory $T$ is \textit{trivial} if a set of tuples is (forking) independent over a base $D$ whenever they are pairwise independent over $D$. When $T$ is strongly minimal, this is equivalent to $T$ having \textit{disintegrated} pregeometry, i.e., where $\text{acl}(A)=\bigcup_{a\in A}\text{acl}(a)$. As such, it is a central notion of geometric simplicity closely connected to the nonexistence of a definable group, as per Zilber's trichotomy principle.

In \cite{goode}, Goode extended the notion of triviality to that of \textit{$k$-triviality}. Say that a simple theory $T$ is \textit{$k$-trivial} if a set of tuples is independent over a base $D$ whenever any $k+1$ of them are independent over $D$ (so that 1-triviality is equivalent to triviality). Walker \cite[Theorem 8.16]{walker} showed that, among stable theories, $k$-triviality is equivalent to $(k+1)$-distality.

Aside from $k$-triviality, Goode also introduced the stronger notion of \textit{$k$-total triviality}. Recently, Chernikov and Westhead showed the following equivalence.
\patchcmd{\thmhead}{(#3)}{#3}{}{}
\begin{prop}[{\cite[Proposition 4.16]{artemfrancis}}]
    Let $k\in\N^+$. Among stable theories, $k$-total triviality is equivalent to strong $(k+1)$-distality.
\end{prop}
\patchcmd{\thmhead}{#3}{(#3)}{}{}
Prior to the release of their paper, we had independently discovered a proof of this equivalence. Our proof uses different machinery to pass between notions of genericity on the side of non-forking independence (used to define $k$-total triviality) and on the side of indiscernible sequences (used to define strong $(k+1)$-distality), which we hope will shed light on more connections between the two sides, so we include our proof in Section \ref{subsecalternativeproof}.

These two equivalences allow us to transfer facts about $k$-triviality and $k$-total triviality to generate the aforementioned examples for $(k+1)$-distality and strong $(k+1)$-distality.

There is one more question of Walker concerning $k$-distality that one might hope to answer by passing between the equivalence to $(k-1)$-triviality.
\patchcmd{\thmhead}{(#3)}{#3}{}{}
\begin{question}[{\cite[Question 5.2]{walker}}]\label{qnNIP}
    Is there an NIP theory that is $k$-distal for some $k\geq 3$ but not 2-distal?
\end{question}
\patchcmd{\thmhead}{#3}{(#3)}{}{}

Restricting to stable theories, this is equivalent to asking if there is a stable theory that is $k$-trivial for some $k>1$ but not (1-)trivial. This was asked by Goode in \cite[Section 5]{goode}. Goode himself showed that $k$-triviality collapses to (1-)triviality among \textit{superstable} theories \cite[Proposition 3]{goode}. We show that this collapse holds even among simple theories, resolving Goode's question.
\begin{theorem}[Theorem \ref{collapse}]\label{introcollapse}
    Suppose $T$ is simple. For all $k\in\N^+$, $T$ is (1-)trivial if and only if it is $k$-trivial.
\end{theorem}
Thus, if an NIP theory witnesses a positive answer to Question \ref{qnNIP}, it must be unstable.

As a corollary to Theorem \ref{introcollapse}, using results of Palac\'in in \cite{palacin}, we have that stable theories with quantifier elimination in a relational language of bounded arity are trivial, as well as an analogous statement for simple theories.
\begin{cor}[Corollary \ref{corkary}]\label{introcorkary}
    If a stable theory is $k$-ary for some $k\in\N^+$, i.e., it admits quantifier elimination in a relational language of bounded arity, then it must be trivial.

    More generally, if a simple theory is $k$-ary with $(k+1)$-complete amalgamation over models for some $k\in\N^+$, then it must be trivial.
\end{cor}
This complements the following result of Baldwin, Freitag, and Mutchnik.
\patchcmd{\thmhead}{(#3)}{#3}{}{}
\begin{theorem}[{\cite[Section 6]{baldwinfreitagmutchnik}}]
    If a simple theory admits quantifier elimination in a finite relational language, then it must be trivial.
\end{theorem}
\patchcmd{\thmhead}{#3}{(#3)}{}{}
The previous two results leave open the following natural question.
\begin{question}
    If a simple theory admits quantifier elimination in an infinite relational language of bounded arity, must it be trivial?
\end{question}
Throughout this paper, $T$ denotes a complete theory with infinite models, and $\monster\models T$ is $\kappa$-saturated and strongly $\kappa$-homogeneous for some sufficiently large $\kappa$.
\subsection{Structure of the paper}
Section \ref{sec2} concerns forking triviality. After reviewing some basic properties of forking, we collect some facts about $k$-triviality and $k$-total triviality from the literature, and introduce a class of examples of Goode (`labelled free pseudoplanes') that separates the two notions. The section concludes with our result that $k$-triviality collapses to (1-)triviality among simple theories (Theorem \ref{collapse}).

Section \ref{sec3} concerns higher-arity distality. We first collect some facts and open questions about (strong) $k$-distality from the literature. Then, we give four classes of examples of (strongly) $k$-distal theories that are not $k$-ary, and describe another example showing the non-preservation of (strong) $k$-distality under taking reducts. The section concludes with our proof that $k$-total triviality is equivalent to strong $(k+1)$-distality among stable theories (Proposition \ref{equiv}).
\subsection{Acknowledgements}
I would like to thank Amador Martin-Pizarro for his input on strong types and triviality, Paolo Marimon for suggesting the examples in Sections \ref{subsubseckay} and \ref{subsubsecJohnson}, and Vincenzo Mantova for our discussions on Goode's examples. I am also grateful to Artem Chernikov for his invitation to the University of Maryland, our discussions on higher-arity distality, and his comments on an earlier version of this paper. Finally, I would like to thank Julia Wolf for her support throughout this project, and for her feedback on this paper.

\textit{Soli Deo Gloria.}
\subsection{Funding statement}
The author is supported by Julia Wolf's Open Fellowship from the UK Engineering and Physical Sciences Research Council (EP/Z53352X/1).
\section{Triviality of forking}\label{sec2}
In this section, we discuss notions of forking triviality from the literature, culminating in our proof that $k$-triviality collapses to (1-)triviality among simple theories (Theorem \ref{collapse}).
\subsection{Preliminaries on forking}
Let us first collect some basic properties of forking. For an introduction to forking, dividing, simplicity, and stability, we refer the reader to \cite[Chapters 7 and 8]{tentziegler}.
\begin{defn}
    Let $A\subseteq B\subseteq \monster$. For $p\in S(A)$ and $q\in S(B)$ such that $p\subseteq q$, say that $q$ is a \textit{non-forking extension of $p$ over $B$} if $q$ does not fork over $A$.
\end{defn}
\begin{prop}
    Suppose $T$ is simple. Then the following properties of non-forking independence hold for all tuples $a$ and small sets $A,B,C,D\subseteq \monster$.
    \begin{enumerate}[(i)]
        \item \textsc{(Invariance)} If $\sigma$ is an automorphism of $\monster$, then $A\indep_D B$ if and only if $\sigma(A)\indep_{\sigma(D)} \sigma(B)$.
        \item \textsc{(Finite Character)} $A\indep_D B$ if and only if $A\indep_D B_0$ for all finite $B_0\subseteq B$.
        \item \textsc{(Symmetry)} $A\indep_D B$ if and only if $B\indep_D A$.
        \item \textsc{(Existence)} For all (not necessarily small) $D'\subseteq \monster$ such that $D\subseteq D'$, $\tp(a/D)$ has a non-forking extension over $D'$.
        \item \textsc{(Monotonicity and Transitivity)} $AB\indep_D C$ if and only if $A\indep_D C$ and $B\indep_{AD} C$.
        \item \textsc{(Strong base monotonicity)} If $AB\indep_D C$, then $A\indep_D B$ if and only if $A\indep_{CD} B$.
        \item \textsc{(Extension)} If $a\indep_D B$, then there is $a'\equiv_{BD} a$ such that $a'\indep_D BC$.
    \end{enumerate}
\end{prop}
\begin{proof}
    See \cite[Section 7.2]{tentziegler}.
\end{proof}
Given tuples $b_1, ..., b_k$ from $\monster$ and $I\subseteq [k]$, we write $b_I:=(b_i: i\in I)$.
\begin{cor}\label{baseextension}
    Suppose $T$ is simple. Let $a, b_1, ..., b_k, a', b'_1, ..., b'_k$ be tuples from $\monster$, and let $C\subseteq D\subseteq \monster$. If $\tp(a',b'_1, ..., b'_k/D)$ is a non-forking extension of $\tp(a,b_1, ..., b_k/C)$, then for all $I\subseteq [k]$, $a\indep_C b_I$ if and only if $a'\indep_D b'_I$.
\end{cor}
\begin{proof}
    Let $I\subseteq [k]$; note that we have $a'b'_I\indep_C D$. Since $(a',b'_I)\equiv_C (a,b_I)$, by \textsc{Invariance} we have $a\indep_C b_I$ if and only if $a'\indep_C b'_I$, which holds if and only if $a'\indep_D b'_I$ by \textsc{Strong base monotonicity}.
\end{proof}
\begin{defn}
    Let $A\subseteq \monster$ be small and $p\in S(A)$. Say that $p$ is \textit{stationary} if, for all (not necessarily small) $B\subseteq \monster$ such that $A\subseteq B$, the type $p$ has a unique non-forking extension over $B$, which we write as $p\upharpoonright^B$.
    
    Note that $p$ is stationary if and only if it has a unique global non-forking extension.
\end{defn}
\begin{lemma}\label{stationaryinvariant}
    Let $A\subseteq \monster$ be small, and let $p\in S(A)$ be stationary. Then $\globalext{p}$ is $A$-invariant.
\end{lemma}
\begin{proof}
    Suppose $\globalext{p}$ is not $A$-invariant, so there is a formula $\phi(x;y)$ and tuples $b,b'$ from $\monster$ with $b\equiv_A b'$, such that $\phi(x;b)\in p$ and $\phi(x;b')\not\in p$. Let $\sigma$ be an automorphism of $\monster$ that fixes $A$ and sends $b$ to $b'$, and let $q:=\{\psi(x;\sigma(d)): \psi(x;d)\in \globalext{p}\}$. Then $q$ is a global non-forking extension of $p$, distinct from $\globalext{p}$ as $\phi(x;b')\in q\setminus (\globalext{p})$, contradicting stationarity.
\end{proof}
\begin{lemma}\label{typesovermodels}
    If $T$ is stable, then all types over small models are stationary.
\end{lemma}
\begin{proof}
    See \cite[Corollary 8.5.4]{tentziegler}.
\end{proof}
\subsection{Notions of triviality}\label{subsectriviality}
Throughout this subsection, suppose $T$ is simple.

Recall that $T$ is \textit{trivial} if, for all tuples $a, b, c$ and small sets $D$, if $a\indep_D b$, $b\indep_D c$, and  $a\indep_D c$, then $a\indep_D bc$. In \cite{goode}, Goode extended this notion to that of \textit{$k$-triviality}, which we shall build up to define in Definition \ref{defnktrivial}.
\begin{defn}
    Let $D\subseteq \monster$ be small, and let $(a_i: i\in I)$ be a sequence of tuples from $\monster$, indexed by a small linear order $I$. Say that $(a_i: i\in I)$ is \textit{independent over $D$} or \textit{$D$-independent} if, for all $i\in I$, $a_i\indep_D \{a_j: j<i\}$.
\end{defn}
\begin{lemma}\label{independencebasicprop}
    ($T$ is simple.) Let $D\subseteq \monster$ be small, and let $(a_i: i\in I)$ be a sequence of tuples from $\monster$, indexed by a small linear order $I$.
    \begin{enumerate}[(i)]
        \item That $(a_i: i\in I)$ is $D$-independent does not depend on the ordering of $I$.
        \item Let $I$ be partitioned into $J$ and $K$. Then $(a_i: i\in I)$ is $D$-independent if and only if $(a_j: j\in J)$, $(a_k: k\in K)$ are $D$-independent and
        \[\left\lbrace a_j: j\in J\right\rbrace\indep_D \left\lbrace a_k: k\in K\right\rbrace.\]
    \end{enumerate}
\end{lemma}
\begin{proof}
    For (i), see \cite[Corollary 7.2.18]{tentziegler}. Let us show (ii). Suppose firstly that $(a_i: i\in I)$ is $D$-independent, from which one easily deduces that $(a_j: j\in J)$, $(a_k: k\in K)$ are $D$-independent. By \textsc{Finite Character}, it suffices to show that $\{a_j: j\in J\}\indep_D \{a_k: k\in K_0\}$ for all finite $K_0\subseteq K$; this is an easy induction on $|K_0|$. For the converse, by (i) and \textsc{Finite Character}, it suffices to show that $\{a_j: j\in J\setminus J_0\}\indep_D \{a_i: i\in J_0\cup K\}$ for all finite $J_0\subseteq J$; this is an easy induction on $|J_0|$. 
\end{proof}

In light of (i), we will say that a set $A$ of tuples from $\monster$ is \textit{independent over $D$} or \textit{$D$-independent} (for $D\subseteq \monster$ small) if some (equivalently, every) enumeration of $A$ is.
\begin{defn}\label{defnktrivial}
    Let $D\subseteq \monster$ be small. Say that a sequence $a=(a_1, ..., a_k)$ of tuples from $\monster$ is a \textit{$k$-cycle over $D$} if every proper subsequence of $a$ is $D$-independent but $a$ is not $D$-independent.

    For $k\in\N^+$, say that $T$ is \textit{$k$-trivial} if there are no $(k+2)$-cycles over any small $D\subseteq \monster$.
\end{defn}
Note that 1-triviality is equivalent to triviality. For all $k\geq 2$, if $T$ is $(k-1)$-trivial then it is $k$-trivial, so we say that $T$ is \textit{strictly $k$-trivial} if it is $k$-trivial but not $(k-1)$-trivial.

In \cite{goode}, Goode also defined the following notion of \textit{$k$-total triviality}.
\begin{defn}\label{defnktotallytrivial}
    Let $k\in\N^+$. Say that $T$ is \textit{$k$-totally trivial} if, given tuples $a, b_1, ..., b_{k+1}$ from $\monster$ and a small set $D\subseteq\monster$, if $a\indep_D b_{\neq j}$ for all $j\in [k+1]$, then $a\indep_D b_1\cdots b_{k+1}$.
\end{defn}
When $k=1$, we simply say that $T$ is \textit{totally trivial}. As with the triviality hierarchy, for all $k\geq 2$, if $T$ is $(k-1)$-totally trivial then it is $k$-totally trivial, so we say that $T$ is \textit{strictly $k$-totally trivial} if it is $k$-totally trivial but not $(k-1)$-totally trivial.

It will often be convenient to assume that the base $D$ is a small model, and the following lemma says that we can.
\begin{lemma}\label{lemmamodel}
    ($T$ is simple.) In Definitions \ref{defnktrivial} and \ref{defnktotallytrivial}, $D$ may be replaced by a small model.
\end{lemma}
\begin{proof}
    Suppose $T$ is not $k$-trivial, witnessed by a $(k+2)$-cycle $a=(a_1, ..., a_{k+2})$ over $D$. Let $M\models T$ be a small model containing $D$. By \textsc{Existence}, there is a non-forking extension $q\in S(M)$ of $\tp(a/D)$; let $a'=(a'_1, ..., a'_{k+2})\models q$. By Lemma \ref{baseextension}, $a'$ is a $k$-cycle over $M$.
    
    The corresponding claim for $k$-total triviality can be similarly proven.
\end{proof}
It is easy to see that if $T$ is $k$-totally trivial, then it is $k$-trivial. Goode \cite{goode} produced the following examples to show that the converse fails.

For $n\in\N^+$, let $T_n$ be the following $L_n$-theory in $n$ sorts $X_1, ..., X_n$, where $L_n$ consists of function symbols $g_i: X_{i-1}\times X_i\to X_i$ for all $i\in [n]\setminus \{1\}$ (so $L_1=\emptyset$). The axioms of $T_n$ say that, for all $i\in [n]\setminus \{1\}$, $g_i$ induces a free action of $F(X_{i-1})$ --- the free group generated by $X_{i-1}$ --- on $X_i$. That is, for all $i\in [n]\setminus \{1\}$, we have the following.
\begin{enumerate}[(i)]
    \item For all $a\in X_{i-1}$, $g_i(a, \cdot): X_i\to X_i$ is a bijection; we denote this function by $a\ast \cdot$, and its inverse by $a^{-1}\ast \cdot$.
    \item We extend the definition of $\ast$ as follows. For $a_1, ..., a_k\in X_{i-1}$, $\varepsilon_1, ..., \varepsilon_k\in \{-1,1\}$, and $b\in X_i$, set
    \[(a_1^{\varepsilon_1}\cdots a_k^{\varepsilon_k})\ast b:=a_1^{\varepsilon_1}\ast(a_2^{\varepsilon_2}\ast(\cdots(a_k^{\varepsilon_k}\ast b)\cdots)).\]
    This defines an action of $F(X_{i-1})$ on $X_i$. The axioms of $T_n$ say that this action is \textit{free}, that is, every non-trivial element of $F(X_{i-1})$ fixes no point of $X_i$.
\end{enumerate}

We also let $L_\omega:=\bigcup_{n\in\N^+}L_n$ and $T_\omega$ be the $L_\omega$-theory given by $\bigcup_{n\in\omega} T_n$.

The reader should visualise a model $(X_1, ..., X_n)$ of $T_n$ as follows. For each $i\in [n]\setminus \{1\}$, consider $X_i$ as the vertex set of the Schreier graph of the action of $F(X_{i-1})$ on $X_i$. The connected components of this graph correspond to the orbits of this action. Since the action is free, given any two (possibly identical) connected components $Y_1, Y_2$ with base points $b_1\in Y_1$ and $b_2\in Y_2$, any permutation of $X_{i-1}$ induces an isomorphism between $Y_1$ and $Y_2$ sending $b_1$ to $b_2$.

Observe that each Schreier graph is an infinitely branching forest, known to model theorists as a `free pseudoplane'. Thus, we may view $T_n$ as an $n$-sorted `labelled free pseudoplane', where in the Schreier graph on $X_i$ for $i\in [n]\setminus \{1\}$, the edges are labelled by elements of $X_{i-1}$ and their inverses in a `free' way.

With this picture in mind, it is an easy exercise to check that the axioms of $T_n$ determine a complete theory: any two models of $T_n$ are isomorphic by the discussion above. Thus, we will henceforth assume that $T_n$ is complete. Goode showed the following.
\patchcmd{\thmhead}{(#3)}{#3}{}{}
\begin{theorem}[{\cite[Section 2]{goode}}]\label{goodeclassification}
    For all $n\in\N^+$, the theory $T_n$ is superstable, trivial, and strictly $2^{n-1}$-totally trivial. Hence, the theory $T_\omega$ is superstable, trivial, and not $k$-totally trivial for any $k\in\N^+$.
\end{theorem}

Thus, $k$-triviality and total $k$-triviality are not equivalent, even among superstable theories. However, Goode showed the following collapse in superstable theories with finite $U$-ranks.
\begin{prop}[{\cite[Sections 1 and 2]{goode}}]\label{collapsefiniterank}
    Let $T$ be a superstable theory with finite $U$-ranks. Then the following are equivalent.
    \begin{enumerate}[(i)]
        \item $T$ is trivial.
        \item $T$ is totally trivial.
        \item $T$ is $k$-trivial for some $k\in\N^+$.
        \item $T$ is $k$-totally trivial for some $k\in\N^+$.
    \end{enumerate}
\end{prop}

Without the assumption of finite $U$-ranks, conditions (i) and (iii) are still equivalent. (Indeed, observe that, as $n$ increases, $T_n$ goes up the total triviality hierarchy but remains (1-)trivial.)
\begin{prop}[{\cite[Proposition 3]{goode}}]
    Let $T$ be superstable. For all $k\in\N^+$, $T$ is (1-)trivial if and only if it is $k$-trivial.
\end{prop}
Goode asks if the assumption of superstability can be weakened to stability.
\begin{question}[{\cite[Section 5]{goode}}]\label{goodequestionbody}
    Among stable theories, is (1-)triviality equivalent to $k$-triviality for all $k\in\N^+$?
\end{question}
We resolve this question positively in Theorem \ref{collapse}, even among simple theories.

Towards this, recall the following classical characterisation of dividing in simple theories. Given a small set $D\subseteq \monster$, we say that a sequence of tuples from $\monster$ is a \textit{Morley sequence over $D$} if it is $D$-indiscernible and $D$-independent.
\patchcmd{\thmhead}{(#3)}{#3}{}{}
\begin{fact}[{\cite[Proposition 7.2.14]{tentziegler}}]\label{morley}
    ($T$ is simple.) Let $D\subseteq \monster$ be small and $\pi(x,y)$ be a partial type over $D$. Suppose there is a Morley sequence $(b_i: i<\omega)$ over $D$ such that $\{\pi(x;b_i):i<\omega\}$ is consistent. Then $\pi(x,b_0)$ does not divide over $D$.
\end{fact}
\patchcmd{\thmhead}{#3}{(#3)}{}{}
We will need the following Ramsey fact, attributed to Shelah.
\patchcmd{\thmhead}{(#3)}{#3}{}{}
\begin{fact}[{\cite[Lemma 7.2.12]{tentziegler}}]\label{extraction}
    For all small $D\subseteq \monster$, there is a cardinal $\lambda$ such that the following holds. Let $(b_i: i<\lambda)$ be a sequence of tuples from $\monster$. Then there is a $D$-indiscernible sequence $(c_j: j<\omega)$ such that, for all $j_1<\cdots<j_n<\omega$, there are $i_1<\cdots<i_n<\lambda$ such that $b_{i_1}\cdots b_{i_n}\equiv_D c_{j_1}\cdots c_{j_n}$.
\end{fact}
\patchcmd{\thmhead}{#3}{(#3)}{}{}
We are ready to prove our main result.
\begin{theorem}\label{collapse}
    ($T$ is simple.) For all $k\in\N^+$, $T$ is (1-)trivial if and only if it is $k$-trivial.
\end{theorem}
\begin{proof}
    We will show that, for all $k\geq 1$, if $T$ is not $k$-trivial then it is not $(k+1)$-trivial.

    Fix $k\geq 1$, and suppose $T$ is $(k+1)$-trivial but not $k$-trivial. Thus, we may fix a $(k+2)$-cycle $(a, b_1, ..., b_{k+1})$ over some small $D\subseteq \monster$. Since $a\not\indep_D b_1\cdots b_{k+1}$, there is an $L(D)$-formula $\phi(x; y_1, ..., y_{k+1})$ such that $\models \phi(a; b_1, ..., b_{k+1})$ and $\phi(x;b_1, ..., b_{k+1})$ divides over $D$.

    Let $\lambda$ be a cardinal given by Fact \ref{extraction} for the set $Da\subseteq \monster$. We prove the following claim.
    \begin{claim*}
        There are tuples $b^{(i)}_1, ..., b^{(i)}_{k+1}$ for $i<\lambda$, where $(b_1^{(0)}, ..., b^{(0)}_{k+1})=(b_1, ..., b_{k+1})$, such that the following conditions hold for all $j<\lambda$.
        \begin{enumerate}[(i)]
            \item The sequence $(b^{(i)}_1, ..., b^{(i)}_{k+1}: i<j+1)$ is $D$-independent (where, for each $i<j+1$ and $s\in [k+1]$, $b^{(i)}_s$ is treated as an individual tuple).
            \item For all $e<j+1$ and $s\in [k+1]$,
            \[\left\lbrace b^{(e)}_r: r\in [k+1]\setminus\{s\}\right\rbrace\underset{D}{\indep} a \left\lbrace b^{(i)}_1, ..., b^{(i)}_{k+1}: i<j+1, i\neq e\right \rbrace.\]
            \item For all $i<j+1$, $(b^{(i)}_1, ..., b^{(i)}_{k+1})\equiv_{Da} (b^{(0)}_1, ..., b^{(0)}_{k+1})$.
        \end{enumerate}
    \end{claim*}
    Before proving the claim, let us show how it can be used to prove the statement. Fix tuples $b^{(i)}_1, ..., b^{(i)}_{k+1}$ for $i<\lambda$ as in the claim. By Fact \ref{extraction}, there is a $Da$-indiscernible sequence $((c^{(j)}_1, ..., c^{(j)}_{k+1}): j<\omega)$ --- where, for each $j<\omega$, $(c^{(j)}_1, ..., c^{(j)}_{k+1})$ is treated as one element of the sequence --- such that, for all $n<\omega$, there are $i_1<\cdots<i_n<\lambda$ such that
    \[\left(b^{(i_1)}_1, ..., b^{(i_1)}_{k+1}, ..., b^{(i_n)}_1, ..., b^{(i_n)}_{k+1}\right)\equiv_{Da} \left(c^{(1)}_1, ..., c^{(1)}_{k+1}, ..., c^{(n)}_1, ..., c^{(n)}_{k+1}\right).\]
    In particular, by condition (i) and \textsc{Invariance}, $((c^{(j)}_1, ..., c^{(j)}_{k+1}): j<\omega)$ is $D$-independent, so is a Morley sequence over $D$. By condition (iii), $\models \phi(a; c^{(j)}_1, ..., c^{(j)}_{k+1})$ for all $j<\omega$, and so by Fact \ref{morley} we have that $\phi(x;c^{(0)}_1, ..., c^{(0)}_{k+1})$ does not divide over $D$. But now $(c^{(0)}_1, ..., c^{(0)}_{k+1})\equiv_D (b_1, ..., b_{k+1})$, and so $\phi(x;b_1, ..., b_{k+1})$ does not divide over $D$, which is a contradiction.

    Let us now prove the claim, constructing $b^{(j)}_1, ..., b^{(j)}_{k+1}$ by induction on $j<\lambda$. When $j=0$, set $(b_1^{(0)}, ..., b^{(0)}_{k+1})=(b_1, ..., b_{k+1})$. Since $(b_1, ..., b_{k+1})$ is $D$-independent, condition (i) is satisfied. For all $s\in [k+1]$, we have that $(a, b_1, ..., b_{s-1}, b_{s+1}, ..., b_{k+1})$ is $D$-independent, so condition (ii) is satisfied. Condition (iii) is trivially satisfied.

        Suppose now that $j>0$ and we have defined $b^{(i)}_1, ..., b^{(i)}_{k+1}$ for $i<j$. Let $q$ be a non-forking extension of $\tp(b^{(0)}_1, ..., b^{(0)}_{k+1}/Da)$ over $Da\cup\{b_1^{(i)}, ..., b^{(i)}_{k+1}: i<j\}$, and let $(b_1^{(j)}, ..., b^{(j)}_{k+1})\models q$. Let us check that the three conditions are satisfied.

        Condition (iii) is trivially satisfied. For condition (ii), let us first consider the case that $e=j$; fix $s\in [k+1]$. Firstly, observe that since $(b_1^{(j)}, ..., b^{(j)}_{k+1})\models q$ and $q$ does not fork over $Da$,
        \[\left\lbrace b^{(j)}_r: r\in [k+1]\setminus\{s\}\right\rbrace\underset{Da}{\indep} \left\lbrace b^{(i)}_1, ..., b^{(i)}_{k+1}: i< j\right \rbrace.\]
        Since $\{b^{(0)}_r: r\in [k+1]\setminus\{s\}\}\indep_D a$, we have $\{b^{(j)}_r: r\in [k+1]\setminus\{s\}\}\indep_D a$ by \textsc{Invariance}. By \textsc{Transitivity} and \textsc{Symmetry}, we obtain
        \[\left\lbrace b^{(j)}_r: r\in [k+1]\setminus\{s\}\right\rbrace\underset{D}{\indep} a\left\lbrace b^{(i)}_1, ..., b^{(i)}_{k+1}: i< j\right \rbrace.\]
        
        This establishes condition (ii) in the case $e=j$. Suppose now $e<j$, and fix $s\in [k+1]$. Firstly, observe that since $(b_1^{(j)}, ..., b^{(j)}_{k+1})\models q$ and $q$ does not fork over $Da$,
        \[\left\lbrace b^{(j)}_1, ..., b^{(j)}_{k+1}\right\rbrace\underset{Da}{\indep} \left\lbrace b^{(i)}_1, ..., b^{(i)}_{k+1}: i< j, i\neq e\right \rbrace\cup \left\lbrace b^{(e)}_r: r\in [k+1]\setminus \{s\}\right \rbrace,\]
        and so by \textsc{Monotonicity},
        \[\left\lbrace b^{(j)}_1, ..., b^{(j)}_{k+1}\right\rbrace\underset{Da\left\lbrace b^{(i)}_1, ..., b^{(i)}_{k+1}: i< j, i\neq e\right\rbrace}{\indep} \left\lbrace b^{(e)}_r: r\in [k+1]\setminus \{s\}\right\rbrace.\]
        By the induction hypothesis,
         \[\left\lbrace b^{(e)}_r: r\in [k+1]\setminus\{s\}\right\rbrace\underset{D}{\indep} a\left\lbrace b^{(i)}_1, ..., b^{(i)}_{k+1}: i< j, i\neq e\right \rbrace.\]
         By \textsc{Transitivity} and \textsc{Symmetry}, we obtain
        \[\left\lbrace b^{(e)}_r: r\in [k+1]\setminus\{s\}\right\rbrace\underset{D}{\indep} a \left\lbrace b^{(i)}_1, ..., b^{(i)}_{k+1}: i<j+1, i\neq e\right \rbrace,\]
        establishing condition (ii) in the case $e<j$.
        
        It remains to verify condition (i); suppose it fails. By \textsc{Finite Character}, there is a finite subset $I\subseteq j+1$ such that $(b^{(i)}_1, ..., b^{(i)}_{k+1}: i\in I)$ is not $D$-independent; fix such $I$ with $|I|$ minimal. Since $(b^{(0)}_1, ..., b^{(0)}_{k+1})$ is $D$-independent and $(b_1^{(i)}, ..., b^{(i)}_{k+1})\equiv_D (b^{(0)}_1, ..., b^{(0)}_{k+1})$ for all $i<j+1$, we have by \textsc{Invariance} that $(b^{(i)}_1, ..., b^{(i)}_{k+1})$ is $D$-independent for all $i<j+1$. In particular, $|I|\geq 2$.
        
        We will show that every proper subsequence of $(b^{(i)}_1, ..., b^{(i)}_{k+1}: i\in I)$ is $D$-independent. Fix $e\in I$ and $s\in [k+1]$. By condition (ii), which we have established,
        \[\left\lbrace b^{(e)}_r: r\in [k+1]\setminus\{s\}\right\rbrace\underset{D}{\indep} \left\lbrace b^{(i)}_1, ..., b^{(i)}_{k+1}: i\in I\setminus \{e\}\right \rbrace.\]
        By minimality of $|I|$, $(b^{(i)}_1, ..., b^{(i)}_{k+1}: i\in I\setminus \{e\})$ is $D$-independent. Since $(b^{(e)}_r: r\in [k+1]\setminus\{s\})$ is also $D$-independent, by Lemma \ref{independencebasicprop}(ii), $(b^{(i)}_1, ..., b^{(i)}_{k+1}: i\in I)\setminus \{b^{(e)}_s\}$ is $D$-independent.
        
        Thus, every proper subsequence of $(b^{(i)}_1, ..., b^{(i)}_{k+1}: i\in I)$ is $D$-independent. Since $(k+1)|I|\geq 2(k+1)\geq k+3$ and $T$ is $(k+1)$-trivial, $(b^{(i)}_1, ..., b^{(i)}_{k+1}: i\in I)$ must be $D$-independent, which is a contradiction. Thus, condition (i) is satisfied, concluding the proof.
\end{proof}
\subsection{Arity}
Recall that a theory is \textit{$k$-ary} if it admits quantifier elimination in a $k$-ary relational language, or equivalently, if every formula $\phi(x_1, ..., x_l)$ is equivalent to a Boolean combination of formulas involving at most $k$ of the variables $x_1, ..., x_l$. Palac\'in proved the following proposition, which says that $(k+1)$-ary theories are (totally) $k$-trivial under suitable assumptions.
\patchcmd{\thmhead}{(#3)}{#3}{}{}
\begin{prop}[{\cite[Section 3]{palacin}}]\label{palacin}
    Let $k\in\N^+$, and suppose $T$ is $(k+1)$-ary.
    \begin{enumerate}[(i)]
        \item If $T$ is simple and has $(k+2)$-complete amalgamation over models (see \cite[Definition 2.2]{palacin}), then $T$ is $k$-trivial. If $T$ is also supersimple, then $T$ is trivial.
        \item If $T$ is stable, then $T$ is $k$-totally trivial (and hence $k$-trivial).
    \end{enumerate}
\end{prop}
\patchcmd{\thmhead}{#3}{(#3)}{}{}
We can now strengthen this as follows.
\begin{cor}\label{corkary}
    \begin{enumerate}[(i)]
        \item Suppose $T$ is simple and, for some $k\in\N^+$, $T$ is $k$-ary and has $(k+1)$-complete amalgamation over models. Then $T$ is trivial.
        \item If $T$ is stable and $k$-ary for some $k\in\N^+$ (i.e., admits quantifier elimination in a relational language of bounded arity), then $T$ is trivial.
    \end{enumerate}
\end{cor}
\section{Higher-arity distality}\label{sec3}
As mentioned in the introduction, there is a close connection between notions of forking triviality and notions of higher-arity distality. In this section, we discuss the latter. There are important gaps in the literature on examples of theories satisfying higher-arity distality, and the main goal of this section is to fill them.

Let us first recall Simon's definition of distality in \cite{simondistal}.
\begin{defn}
    Say that $T$ is \textit{distal} if either of the following equivalent conditions hold.
    \begin{enumerate}[(i)]
        \item \textit{(`Internal characterisation')} Let $I_0, I_1, I_2$ be infinite sequences of tuples from $\monster$. Let $a_1, a_2$ be tuples from $\monster$. If $I_0+I_1+a_2+I_2$ and $I_0+a_1+I_1+I_2$ are indiscernible, then $I_0+a_1+I_1+a_2+I_2$ is indiscernible.
        \item \textit{(`External characterisation')} Let $I, J$ be infinite sequences of tuples from $\monster$. Let $a$ be a tuple from $\monster$, and let $B\subseteq \monster$ be small. If $I+a+J$ is indiscernible and $I+J$ is indiscernible over $B$, then $I+a+J$ is indiscernible over $B$.
    \end{enumerate}
\end{defn}
\begin{remark}
\begin{enumerate}[(i)]
    \item By compactness, the sequences in the previous definition may be assumed to be dense and/or without endpoints.
    \item In \cite{simondistal}, Simon defined distality via the internal characterisation, and then showed that the external characterisation was equivalent. That the external implies the internal is easy, but the converse implication is non-trivial. In particular, Simon's proof relied on the assumption that $T$ is NIP (an assumption made throughout the paper). It was only in a much later paper \cite{walker} that a proof, due to Chernikov, appeared of the fact that the internal characterisation already implies NIP, so the assumption was in fact superfluous.
\end{enumerate}
\end{remark}

Walker \cite{walker} generalised the two characterisations of distality to higher arity as follows. Recall that given tuples $b_1, ..., b_k$ from $\monster$ and $I\subseteq [k]$, we write $b_I:=(b_i: i\in I)$; for $i\in [k]$, we also write $b_{\neq i}:=b_{[k]\setminus \{i\}}$.
\begin{defn}\label{defnkdistal}
    Let $k\in\N^+$. Say that $T$ is \textit{$k$-distal} if the following holds.

    Let $I_0, ..., I_{k+1}$ be infinite sequences of tuples from $\monster$. Let $a_0, ..., a_k$ be tuples from $\monster$. If $I_0+a_0+I_1+a_1+\cdots+I_{m-1}+I_m+\cdots+a_k+I_{k+1}$ is indiscernible for all $m\in [k+1]$, then $I_0+a_0+I_1+a_1+\cdots+I_k$ is indiscernible.
\end{defn}
\begin{defn}\label{defnstronglykdistal}
    Let $k\in\N^+$. Say that $T$ is \textit{strongly $k$-distal} if the following holds.

    Let $I, J$ be infinite sequences of tuples from $\monster$. Let $a, b_1, ..., b_k$ be tuples from $\monster$. If $I+a+J$ is indiscernible over $b_{\neq m}$ for all $m\in [k]$ and $I+J$ is indiscernible over $b_1\cdots b_k$, then $I+a+J$ is indiscernible over $b_1\cdots b_k$.
\end{defn}
\begin{remark}\label{rmkrename}
    \begin{enumerate}[(i)]
        \item By compactness, in Definitions \ref{defnkdistal} and \ref{defnstronglykdistal}, the sequences may be assumed to be dense and/or without endpoints.
        \item In Definition \ref{defnstronglykdistal}, the tuples $b_1, ..., b_{k+1}$ may be replaced by small sets.
    \end{enumerate}
\end{remark}

It is easy to check that, for all $k\geq 2$, if $T$ is (strongly) $(k-1)$-distal then it is (strongly) $k$-distal, so we say that $T$ is \textit{strictly (strongly) $k$-distal} if it is (strongly) $k$-distal but not (strongly) $(k-1)$-distal.

It is also easy to check that, for all $k\in\N^+$, if $T$ is strongly $k$-distal then it is $k$-distal. The converse holds for $k=1$, as both notions are equivalent to distality, but was open for $k\geq 2$.
\patchcmd{\thmhead}{(#3)}{#3}{}{}
\begin{question}[{\cite[Question 5.2]{walker}}]\label{questioncoincide}
    Let $k\geq 2$. Do $k$-distality and strong $k$-distality coincide?
\end{question}
\patchcmd{\thmhead}{#3}{(#3)}{}{}
This was resolved negatively by Chernikov and Westhead in \cite{artemfrancis}. Their argument uses the fact that $k$-distality and strong $k$-distality respectively correspond to $(k-1)$-triviality and $(k-1)$-total triviality among stable theories. The latter correspondence was proven by them, and the former by Walker.
\patchcmd{\thmhead}{(#3)}{#3}{}{}
\begin{fact}[{\cite[Theorem 8.16]{walker}}]\label{walkerequiv}
    Let $T$ be a stable theory. For all $k\in\N^+$, $T$ is $k$-trivial if and only if it is $(k+1)$-distal.
\end{fact}
\patchcmd{\thmhead}{#3}{(#3)}{}{}
\patchcmd{\thmhead}{(#3)}{#3}{}{}
\begin{fact}[{\cite[Proposition 4.16]{artemfrancis}}]\label{artemfrancis}
    Let $T$ be a stable theory. For all $k\in\N^+$, $T$ is $k$-totally trivial if and only if it is strongly $(k+1)$-distal.
\end{fact}
Before the release of \cite{artemfrancis}, we independently discovered a proof of Fact \ref{artemfrancis}, which we include in Section \ref{subsecalternativeproof}.

Using Facts \ref{walkerequiv} and \ref{artemfrancis}, Chernikov and Westhead use Goode's examples from Section \ref{subsectriviality} to give a negative answer to Question \ref{questioncoincide}.
\begin{cor}[{\cite[Corollary 4.23]{artemfrancis}}]
    For all $n\in\N^+$, the theory $T_n$ is superstable, strictly 2-distal, strictly strongly $(2^{n-1}+1)$-distal. The theory $T_\omega$ is superstable, strictly 2-distal, and not strongly $k$-distal for any $k\in\N^+$.
\end{cor}
Note that the theory $T_\omega$ is $k$-distal for all $k\geq 2$, thus separating $k$-distality from strong $k$-distality for all $k\geq 2$.

Other than Goode's theories, the only examples in the literature of non-distal (strongly) $k$-distal theories appear in Walker's work \cite{walker, walkerthesis}. Curiously, he is unable to find NIP theories that are strictly $k$-distal for some $k\geq 3$.
\begin{question}[{\cite[Question 5.2]{walker}}]\label{questionNIP}
    Let $k\geq 3$. Is there an NIP theory that is strictly $k$-distal?
\end{question}
\patchcmd{\thmhead}{#3}{(#3)}{}{}
Note that Goode's examples are NIP and strictly \textit{strongly} $k$-distal for some $k\geq 3$, but they are (strictly) 2-distal. Using Theorem \ref{collapse} and Fact \ref{walkerequiv}, we partially answer Question \ref{questionNIP}.
\begin{cor}
    Suppose $T$ is stable. For all $k\geq 2$, $T$ is 2-distal if and only if it is $k$-distal.

    In particular, if there is an NIP theory that is strictly $k$-distal for some $k\geq 3$ (thus resolving Question \ref{questionNIP} positively), it must be unstable.
\end{cor}
There is another important gap in the literature on (strongly) $k$-distal examples.
\subsection{Arity}\label{subsecarity}
Recall that a theory is \textit{$k$-ary} if it admits quantifier elimination in a $k$-ary relational language. All examples in \cite{walker} of non-distal (strongly) $k$-distal theories are $k$-ary. It is easy to see that every $k$-ary theory is strongly $k$-distal; this was already shown in \cite{walker}, but we include a proof for completeness.
\patchcmd{\thmhead}{(#3)}{#3}{}{}
\begin{prop}[{\cite[Corollary 4.14]{walker}}]\label{k-aryisstronglyk-distal}
    Let $k\in \N^+$. If $T$ is $k$-ary, then $T$ is strongly $k$-distal (and hence $k$-distal).
\end{prop}
\patchcmd{\thmhead}{#3}{(#3)}{}{}
\begin{proof}
    Let $I, J$ and $a, b_1, ..., b_k$ be as in the hypothesis of Definition \ref{defnstronglykdistal}. Let $\phi(x, y_1, ..., y_k)\in L(I_0J_0)$ for some finite $I_0\subseteq I$, $J_0\subseteq J$. Fixing any $a'\in I$ occurring after every element of $I_0$, we show that $\phi(a, b_1, ..., b_k)$ is equivalent to $\phi(a', b_1, ..., b_k)$.
    
    Since $T$ is $k$-ary, $\phi(x, y_1, ..., y_k)$ is equivalent in $T$ to a Boolean combination of some \linebreak$\psi_1(x, y_1, ..., y_k), ..., \psi_l(x, y_1, ..., y_k)\in L(I_0J_0)$, where each $\psi_j$ either has no $x$-dependence or no $y_i$-dependence for some $i\in [k]$. Write $\tau(\psi_1, ..., \psi_l)$ for this Boolean combination. We have that $\phi(a, b_1, ..., b_k)$ is equivalent to $\tau(\psi_1(a, b_1, ..., b_k), ..., \psi_l(a, b_1, ..., b_k))$, which is in turn equivalent to $\tau(\psi_1(a', b_1, ..., b_k), ..., \psi_l(a', b_1, ..., b_k))$ by indiscernibility. But this is equivalent to $\phi(a', b_1, ..., b_k)$ as required.
\end{proof}
As mentioned in the introduction, to gain a deeper understanding of (strong) $k$-distality, it is essential to find non-distal (strongly) $k$-distal theories that are not $k$-ary. We give four classes of examples; we are grateful to Paolo Marimon for suggesting the second and third ones to us. It will be convenient to talk about the precise \textit{arity} of a theory.
\begin{defn}
    The \textit{arity} of the theory $T$, written $\ar(T)$, is the least $k\in\omega$ such that $T$ is $k$-ary, if such $k$ exists, and $\omega$ otherwise.
\end{defn}
So, for $k\in\N^+$, $T$ is $k$-ary if and only if $\ar(T)\leq k$. We seek (strongly) $k$-distal theories $T$ with $\ar(T)> k$.
\subsubsection{Goode's labelled free pseudoplanes}
Our first example has already been introduced: they are Goode's labelled free pseudoplanes $T_n$ from Section \ref{subsectriviality}. Recall that $T_n$ is superstable, (1-)trivial, and strictly $2^{n-1}$-totally trivial, and hence strictly 2-distal and strictly strongly $(2^{n-1}+1)$-distal by Facts \ref{walkerequiv} and \ref{artemfrancis}. Let $(X_1, X_2, ...)\models T_\omega$ be $\kappa$-saturated and strongly $\kappa$-homogeneous for some sufficiently large $\kappa$.
\begin{prop}\label{Tnnotkary}
    For all $n\in \N^+$, there is a formula $\phi_n(\bar{x})=\phi_n(x_1, ..., x_{2^n})\in L_n$ and a sequence $\bar{a}=(a_1, ..., a_{2^n})$ of elements of $X_n$, such that $\bar{a}\models \phi_n(\bar{x})$ and
    \[\bigcup_{i=1}^{2^n}\tp_{L_n}(\bar{a}_{\neq i}/\emptyset)\not\vdash \phi_n(\bar{x}).\]
    In particular, $\phi_n$ is not a Boolean combination of (parameter-free) $L_n$-formulas in less than $2^n$ variables, and so $\ar(T_n)>2^n-1$.
\end{prop}
\begin{proof}
    Induct on $n\in \N^+$. For $n=1$, let $\phi_1(x_1, x_2):= (x_1=x_2)$; it is clear that, for all $a\in X_1$, $(a,a)\models \phi_1(x_1, x_2)$ satisfies the statement.

    Suppose, for some $n\in \N^+$, that such $\phi_n(\bar{x})=\phi_n(x_1, ..., x_{2^n})$ and $\bar{a}=(a_1, ..., a_{2^n})$ are given. Let $\phi_{n+1}(y_1, ..., y_{2^{n+1}})$ be the formula
    \[\exists x_1, ..., x_{2^n}\left(\phi_n(x_1, ..., x_{2^n})\wedge\bigwedge_{i=1}^{2^n}y_{i+2^n}=x_i\ast y_i\right).\]
    Fix any $b_1, ..., b_{2^n}\in X_{n+1}$ in pairwise distinct orbits. For $i\in [2^n]$, let $b_{i+2^n}:=a_i\ast b_i$. Writing $\bar{b}:=(b_1, ..., b_{2^{n+1}})$, we have that $\bar{b}\models \phi_{n+1}(\bar{y})$. Since
    \[\bigcup_{i=1}^{2^n}\tp_{L_n}(\bar{a}_{\neq i}/\emptyset)\not\vdash \phi_n(\bar{x}),\]
    we may fix $\bar{a'}=(a'_1, ..., a'_{2^n})$ such that $\bar{a'}$ satisfies the left hand side above but not $\phi_n(\bar{x})$. Now, for $i\in [2^n]$, let $b'_i:=b_i$ and $b'_{i+2^n}:=a'_i\ast b'_i$; write $\bar{b'}:=(b_1, ..., b_{2^{n+1}})$. Then $\bar{b'}\not\models \phi_{n+1}(\bar{y})$: since the $F(X_n)$-action is free, the unique tuple $\bar{x}=(x_1, ..., x_{2n})$ such that $y_{i+2^n}=x_i\ast y_i$ for all $i\in [2^n]$ is $\bar{a'}$, and $\bar{a'}\not\models \phi_n(\bar{x})$. We claim that
    \[b'\models \bigcup_{i=1}^{2^{n+1}}\tp_{L_{n+1}}(\bar{b}_{\neq i}/\emptyset).\]
    
    Indeed, fix $j\in [2^{n+1}]$; we claim that $\tp_{L_{n+1}}(\bar{b}_{\neq j}/\emptyset)=\tp_{L_{n+1}}(\bar{b'}_{\neq j}/\emptyset)$. Let $i\in [2^n]$ be such that $i\equiv j\mod 2^n$. Since $\bar{a'}_{\neq i}\models \tp_{L_n}(\bar{a}_{\neq i}/\emptyset)$, there is an $L_n(\emptyset)$-automorphism $\tau: X_1\cup\cdots\cup X_n\to X_1\cup\cdots\cup X_n$ sending $\bar{a}_{\neq i}$ to $\bar{a'}_{\neq i}$. Now, in $\bar{b}_{\neq j}$, the entries $b_r$ and $b_{r+2^n}$ belong to the same orbit for all $r\in [2^n]\setminus \{i\}$, and every other pair of entries belongs to different orbits; the same is true for $\bar{b'}_{\neq j}$. For all $r\in [2^n]\setminus \{i\}$, $a_r\ast b_r=b_{r+2^n}$ and $\tau(a_r)\ast b'_r=a'_r\ast b'_r=b'_{r+2^n}$; it is an easy exercise to extend $\tau$ to an $L_{n+1}(\emptyset)$-automorphism $\sigma: X_1\cup\cdots\cup X_{n+1}\to X_1\cup\cdots\cup X_{n+1}$ sending $\bar{b}_{\neq j}$ to $\bar{b'}_{\neq j}$, showing that $\tp_{L_{n+1}}(\bar{b}_{\neq j}/\emptyset)=\tp_{L_{n+1}}(\bar{b'}_{\neq j}/\emptyset)$ as required.
    \end{proof}
For $n\geq 2$, we have that $T_n$ is strictly 2-distal and strictly strongly $(2^{n-1}+1)$-distal, with $\ar(T_n)>2^n-1\geq 2^{n-1}+1>2$.
\subsubsection{Universal homogeneous kay-graph}\label{subsubseckay}
Fix $k\in\N^+\setminus\{1\}$. Let $\text{RG}_k$ be the theory of the random $k$-uniform hypergraph (see \cite[Section 5.1]{walker} for a precise description of this theory), in the language consisting of a unique $k$-ary relation symbol. Let $(V,E)$ be any sufficiently saturated model of $\text{RG}_k$. Define a reduct $(V,R)$ of this model as follows, where $R$ is a $(k+1)$-ary relation symbol. For $v_1, ..., v_{k+1}\in V$, let $V\models R(v_1, ..., v_{k+1})$ if and only if
\[\abs{\left\lbrace i\in [k+1]: V\models E(v_{\neq i})\right\rbrace}\equiv 1\mod 2,\]
where $E(v_{\neq i})$ is shorthand for $E(v_1, ..., v_{i-1}, v_{i+1}, ..., v_{k+1})$. Equivalently,
\[\mathbbm{1}_R(v_1, ..., v_{k+1})\equiv \sum_{i=1}^{k+1}\mathbbm{1}_E(v_{\neq i})\mod 2.\]

Let $T_k:=\Th((V,R))$; this is known as the theory of the \textit{universal homogeneous kay-graph}. By \cite[Proposition 3.3]{kaygraphs}, $T_k$ admits quantifier elimination, so $\ar(T_k)=k+1$.
\begin{prop}
    The theory $T_k$ is strictly (strongly) $k$-distal.
\end{prop}
\begin{proof}
Firstly, we show that $T_k$ is not $(k-1)$-distal. Indeed, let $v_1, ..., v_k\in V$ and $I_0, ..., I_k$ be infinite sequences of elements of $V$, such that:
\begin{enumerate}[(i)]
    \item The sequence $J:=I_0+v_1+I_1+v_2+\cdots+I_k$ consists of pairwise distinct elements;
    \item $V\models \neg E(v_1, ..., v_k)$;
    \item $V\models E(x_1, ..., x_k)$ for every subsequence $(x_1, ..., x_k)$ of $J$ distinct from $(v_1, ..., v_k)$.
\end{enumerate}
The reader is invited to check that $J$ is not indiscernible but $J\setminus \{v_i\}$ is indiscernible for all $i\in [k]$, showing that $T_k$ is not $(k-1)$-distal.

Next, we show that $T_k$ is strongly $k$-distal. By quantifier elimination, it suffices to check that if $I,J$ are infinite sequences of singletons and $a, b_1, ..., b_k$ are singletons, such that $I+a+J$ is indiscernible over $b_{\neq i}$ for all $i\in [k]$ and $I+J$ is indiscernible over $b_1\cdots b_k$, then $I+a+J$ is indiscernible over $b_1\cdots b_k$. By quantifier elimination, it suffices to fix $x\in I$ and show that $\mathbbm{1}_R(a, b_1, ..., b_k)=\mathbbm{1}_R(x, b_1, ..., b_k)$. Observe that for all $y\in a+J$,
\[\mathbbm{1}_R(y, b_1, ..., b_k)=\mathbbm{1}_R(x, b_1, ..., b_k)\Leftrightarrow \sum_{i=1}^k \mathbbm{1}_E(y, b_{\neq i})\equiv \sum_{i=1}^k \mathbbm{1}_E(x, b_{\neq i}) \mod 2,\]
and
\begin{align*}
    &\sum_{i=1}^k \mathbbm{1}_E(y, b_{\neq i})+ \sum_{i=1}^k\mathbbm{1}_E(x, b_{\neq i})\\
    &\equiv \sum_{i=1}^k \mathbbm{1}_E(y, b_{\neq i})+ \sum_{i=1}^k\mathbbm{1}_E(x, b_{\neq i})+2\sum_{1\leq i<j\leq k}\mathbbm{1}_E(x, y, b_{\neq i, j})\mod 2\\
    &\equiv \sum_{i=1}^k\mathbbm{1}_R(x, y, b_{\neq i})\mod 2.
\end{align*}

Fix any $y\in J$. Since $I+J$ is $b_1\cdots b_k$-indiscernible, $\mathbbm{1}_R(y, b_1, ..., b_k)=\mathbbm{1}_R(x, b_1, ..., b_k)$, so the argument above shows that $\sum_{i=1}^k\mathbbm{1}_R(x, y, b_{\neq i})\equiv 0\mod 2$. Since $I+a+J$ is $b_{\neq i}$-indiscernible for all $i\in [k]$, $\mathbbm{1}_R(x, y, b_{\neq i})=\mathbbm{1}_R(x, a, b_{\neq i})$ for all $i\in [k]$, and so $\sum_{i=1}^k\mathbbm{1}_R(x, a, b_{\neq i})\equiv 0\mod 2$. The argument above now shows that $\mathbbm{1}_R(a, b_1, ..., b_k)=\mathbbm{1}_R(x, b_1, ..., b_k)$ as required.
\end{proof}
\subsubsection{Johnson graph}\label{subsubsecJohnson}
For $k\in\N^+$, the \textit{Johnson graph} $\mathfrak{J}(k)$ is the graph on vertex set $\N^{(k)}$ (that is, the set of subsets of $\N$ of size $k$) such that for all $x,y\in \N^{(k)}$, $(x,y)$ is an edge if and only if $|x\cap y|=k-1$. Consider this as a structure in the language of graphs. We describe a finite relational language in which $\mathfrak{J}(k)$ is homogeneous.

For $i,j\in\N$ with $i\geq 2$ and $0\leq j\leq k-1$, let $E_{ij}$ be an $i$-ary relation symbol interpreted as
\[E_{ij}(x_1, ..., x_i):\Leftrightarrow |x_1\cap\cdots \cap x_i|=j,\]
so that $E_{2, k-1}$ is the graph relation.

Note that all such $E_{ij}$ are definable in $\mathfrak{J}(k)$. Indeed, for $j\in\N$ with $0\leq j\leq k-1$, and $x_1, x_2\in\N^{(k)}$,
\[|x_1\cap x_2|=j\Leftrightarrow |x_1\cap x_2|<j+1\wedge \exists y \left(|x_1\cap y|=j+1\wedge |x_2\cap y|=k-1\right),\]
so by downward induction on $j$, $E_{2j}$ is definable in $\mathfrak{J}(k)$ for all $0\leq j\leq k-1$. Now, for $i,j\in\N$ with $i\geq 3$ and $0\leq j\leq k-1$, and $x_1, ..., x_i\in \N^{(k)}$, we claim that for $N:={ki\choose j+1}+1$,
\begin{align*}
    &|x_1\cap\cdots\cap x_i|=j\Leftrightarrow |x_1\cap\cdots\cap x_i|<j+1\\
    &\phantom{adsfdsafhsadfhsadfjkd}\wedge \exists y_1, ..., y_N \left(\bigwedge_{1\leq p<q\leq N}|y_p\cap y_q|=j\wedge \bigwedge_{p=1}^N\bigwedge_{l=1}^i|x_l\cap y_p|=j\right).
\end{align*}
Assuming the claim, for each $i\in\N$ with $i\geq 3$, we have by downward induction on $j$ that $E_{ij}$ is definable in $\mathfrak{J}(k)$ for all $0\leq j\leq k-1$.

Let us argue the claim. The forward direction is straightforward; let $y_1, ..., y_N$ witness the backward direction, and suppose $|x_1\cap\cdots\cap x_i|\leq j-1$. For all $p\in [N]$, $y_p$ intersects each of $x_1, ..., x_i$ in a set of size $j$, so $y_p$ must contain at least $j+1$ elements from $x_1\cup\cdots \cup x_i$. By choice of $N$ and the pigeonhole principle, there are distinct $p,q\in [N]$ such that $y_p, y_q$ contain the same $j+1$ elements from $x_1\cup\cdots \cup x_i$, contradicting the fact that $|y_p\cap y_q|=j$.\footnote{The following quicker proof of the definability of $E_{ij}$, albeit using heavier machinery, was suggested to us by Bertalan Bodor. Let $\mathfrak{J}^+(k)$ be the expansion of $\mathfrak{J}(k)$ by $E_{ij}$ for all $i,j\in\N$ with $i\geq 2$ and $0\leq j\leq k-1$. By \cite[Example 1.2]{thomashypergraphs}, $\mathfrak{J}(k)$ and $\mathfrak{J}^+(k)$ have the same automorphism group ($\Sym(\N)$ acting in the natural way). Both structures are interpretable in $(\N, =)$, so are $\omega$-categorical. Thus, they are reducts of each other.}

Let $L_k:=\{E_{ij}: i,j\in\N, 2\leq i\leq k+1, 0\leq j\leq k-1\}$.
\begin{prop}
    For all $k\in\N^+$, $\mathfrak{J}(k)$ is homogeneous in the language $L_k$.
\end{prop}
\begin{proof}
For $S\subseteq \N^{(k)}$, an injection $\sigma: \bigcup S\to\N$, and a function $\alpha: S\to \N^{(k)}$, say that $\sigma$ \textit{induces} $\alpha$ if for all $x\in S$, $\sigma(x)=\alpha(x)$; equivalently, for all $x\in S$ and $n\in\bigcup S$, $n\in x$ if and only if $\sigma(n)\in \alpha(x)$. Note that such an $\alpha$ must be an $L_k$-isomorphism between $S$ and $\alpha(S)$.

Suppose $S,T\subseteq \N^{(k)}$ are finite and there is an $L_k$-isomorphism $\alpha: S\to T$. We show that there is an injection $\sigma:\bigcup S\to\N$ that induces $\alpha$. Note then that we would be done, as $\sigma$ extends to some $\tau\in \Sym(\N)$, and $\tau$ induces an $L_k$-automorphism of $\N^{(k)}$ that extends $\alpha$.

Define an equivalence relation $\sim$ on $\bigcup S$ by asserting that
\[m\sim n :\Leftrightarrow \text{ for all }x\in S, m\in x \text{ if and only if }n\in x.\]
Let $\mathcal{C}$ be the set of $\sim$-equivalence classes. We have the following.
\begin{enumerate}[(i)]
    \item For all $x\in S$ and $C\in\mathcal{C}$, either $C\subseteq x$ or $C\cap x=\emptyset$.
    \item For all $C\in\mathcal{C}$, there is some $x\in S$ such that $C\subseteq x$, so $1\leq |C|\leq k$.
\end{enumerate}
We will define $\sigma:\bigcup S\to\N$ by defining $\sigma|_C$ for each $C\in\mathcal{C}$. Let $C\in\mathcal{C}$. Then, there is $x_1\in S$ such that $C\subseteq x_1$. For all $n\in x_1\setminus C$, there is $x\in S$ such that either $n\not\in x$ and $C\subseteq x$, or $n\in x$ and $C\cap x=\emptyset$. Thus, as $|x_1\setminus C|\leq k-1$, there are $x_2, ..., x_q\in S$ such that $C=(x_1\cap \cdots\cap x_p)\setminus (x_{p+1}\cup \cdots \cup x_q)$, where $1\leq p\leq q\leq k$. By inclusion-exclusion, $|(x_1\cap \cdots\cap x_p)\setminus (x_{p+1}\cup \cdots \cup x_q)|$ can be expressed in terms of $\{|x_{i_1}\cap\cdots \cap x_{i_l}|: l\in [k], 1\leq i_1<\cdots<i_l\leq q\}$, so as $\alpha$ is an $L_k$-isomorphism, $D:=(\alpha(x_1)\cap \cdots\cap \alpha(x_p))\setminus (\alpha(x_{p+1})\cup \cdots \cup \alpha(x_q))$ has size $|C|$. We set $\sigma|_C$ to be any bijection from $C$ to $D$.

We claim that for all $x\in S$, if $C\subseteq x$ then $\sigma(C)\subseteq \alpha(x)$, and if $C\cap x=\emptyset$ then $\sigma(C)\cap \alpha(x)=\emptyset$. Suppose first that $C\subseteq x$. Then, $C=C\cap x=(x\cap x_1\cap \cdots\cap x_p)\setminus (x_{p+1}\cup \cdots \cup x_q)$. By inclusion-exclusion again, we have that
\[|\sigma(C)|=|C|=|(\alpha(x)\cap \alpha(x_1)\cap \cdots\cap \alpha(x_p))\setminus (\alpha(x_{p+1})\cup \cdots \cup \alpha(x_q))|=|\alpha(x)\cap \sigma(C)|,\]
and so $\sigma(C)\subseteq \alpha(x)$. Suppose instead that $\emptyset=C\cap x=(x\cap x_1\cap \cdots\cap x_p)\setminus (x_{p+1}\cup \cdots \cup x_q)$. By inclusion-exclusion again, we have that
\[0=|(\alpha(x)\cap \alpha(x_1)\cap \cdots\cap \alpha(x_p))\setminus (\alpha(x_{p+1})\cup \cdots \cup \alpha(x_q))|=|\alpha(x)\cap \sigma(C)|,\]
and so $\sigma(C)\cap \alpha(x)=\emptyset$.

We conclude now that $\sigma$ is an injection that induces $\alpha$. For injectivity: certainly $\sigma$ is injective on every $C\in\mathcal{C}$, and if $C,C'\in\mathcal{C}$ are distinct, then without loss of generality there is $x\in S$ such that $C\subseteq x$ and $C'\cap x=\emptyset$, so $\sigma(C)\subseteq \alpha(x)$ and $\sigma(C')\cap \sigma(C)\subseteq \sigma(C')\cap \alpha(x)=\emptyset$. The preceding paragraph now precisely says that $\sigma$ induces $\alpha$.
\end{proof}
We can now collect some properties of $T_k:=\Th(\mathfrak{J}(k))$.
\begin{cor}\label{Johnsonclassification}
    For all $k\in\N^+$, $\ar(T_k)\leq k+1$, and $T_k$ is $\omega$-categorical, $\omega$-stable, and totally trivial --- hence strictly (strongly) 2-distal.
\end{cor}
\begin{proof}
    Since $\mathfrak{J}(k)$ is homogeneous in the language $L_k$, it has quantifier elimination in $L_k$, so $\ar(T_k)\leq k+1$. Moreover, $\mathfrak{J}(k)$ is interpretable in $(\N,=)$, so $T_k$ is $\omega$-categorical and $\omega$-stable. By Proposition \ref{palacin}(ii), we have that $T_k$ is $k$-totally trivial (alternatively, by Corollary \ref{corkary}, we have that $T_k$ is trivial). Since $\omega$-categorical $\omega$-stable theories have finite $U$-ranks \cite[Theorem 5.1]{omegacatomegastable}, by Proposition \ref{collapsefiniterank}, we have that $T_k$ is totally trivial. By Fact \ref{artemfrancis}, we conclude that $T_k$ is strictly (strongly) 2-distal.
\end{proof}
Towards the goal of this subsection, we would like to know if $\ar(T_k)\geq 3$. Observe that $\ar(T_k)\geq l$ if and only if, given $x_1, ..., x_l\in \N^{(k)}$, the values of $|x_{i_1}\cap \cdots \cap x_{i_{l-1}}|$ for $1\leq i_1\leq\cdots\leq i_{l-1}\leq l$ do not determine $|x_1\cap \cdots \cap x_l|$.

Now, since $\mathfrak{J}(1)\cong (\N^{(1)},\neq)$, $\ar(T_1)=2$. However, $\ar(T_2)\geq 3$, witnessed by the triples $(\{0, 1\}, \{0, 2\}, \{1, 2\})$ and $(\{0, 1\}, \{0, 2\}, \{0, 3\})$: in either triple, the sets intersect pairwise in a singleton, but $\{0, 1\}\cap \{0, 2\}\cap\{1, 2\}=\emptyset$ while $\{0, 1\}\cap\{0, 2\}\cap\{0, 3\}=\{0\}$. Since $\ar(T_2)\leq 3$, we conclude that $\ar(T_2)=3$. Observe that $\ar(T_k)\leq \ar(T_{k+1})$ for all $k\in\N^+$, so we have that $\ar(T_k)\geq 3$ for all $k\geq 2$.

The reader may be tempted to conjecture that $\ar(T_k)=k+1$ for all $k\in\N^+$. This is false: indeed, the reader is invited to check that $\ar(T_3)=3$. Working out the precise value of $\ar(T_k)$ appears to be a difficult (or, at least, involved) combinatorial problem, and is outside the scope of this paper. However, it would be instructive to know at least that $\ar(T_k)\to\infty$, as the theories $T_k$ would then constitute examples of strictly (strongly) 2-distal theories that are arbitrarily far from being binary. This is what we show.
\begin{prop}
    For all $l\in\N^+$, there is $k\in\N^+$ such that $\ar(T_k)\geq l$.
\end{prop}
\begin{proof}
    By induction on $p\in [l]$, we construct systems of finite sets $X^{(p)}=(x^{(p)}_1, ..., x^{(p)}_l), Y^{(p)}=(y^{(p)}_1, ..., y^{(p)}_l)$ such that:
    \begin{enumerate}[(i)]
        \item $|x^{(p)}_1\cap \cdots \cap x^{(p)}_l|=1, |y^{(p)}_1\cap \cdots \cap y^{(p)}_l|=0$;
        \item For all $j\in\N^+$ with $l-p<j<l$, there is $r_j\in\N$ such that the intersection of any $j$ distinct sets from $X^{(p)}$ (respectively, $Y^{(p)}$) has size $r_j$.
    \end{enumerate}
    We would then be done: the sets in $X^{(l)}, Y^{(l)}$ all have the same size, say $k$, and the set systems $X^{(l)}, Y^{(l)}$ witness that $\ar(T_k)\geq l$.

    We begin the induction by setting $x^{(1)}_1=\cdots=x^{(1)}_l=\{0\}$ and $y^{(1)}_1=\cdots=y^{(1)}_l=\emptyset$. Suppose for some $i\in [k]$ that we have constructed $X^{(p)}$ and $Y^{(p)}$. Let $r_{l-p}\in\N$ be such that every set in $X^{(p)}$ or $Y^{(p)}$ has size at most $r_{l-p}$. For each $I\in [l]^{(l-p)}$, for $d:=r_{l-p}-|\bigcap_{i\in I} x^{(p)}_i|$, add distinct new elements $n_1, ..., n_d$ to $\bigcap_{i\in I} x^{(p)}_i$ (that is, add $n_1, ..., n_d$ to each $x^{(p)}_i$ for $i\in I$); likewise, for $e:=r_{l-p}-|\bigcap_{i\in I} y^{(p)}_i|$, add $e$ distinct new elements to $\bigcap_{i\in I} y^{(p)}_i$. Let $X^{(p+1)}$ and $Y^{(p+1)}$ collect the resulting sets. Observe then that the intersection of any $l-p$ distinct sets from $X^{(p+1)}$ (respectively, $Y^{(p+1)}$) has size $r_{l-p}$, and we have not changed the sizes of the $j$-fold intersections for $j>l-p$.
\end{proof}

Thus, we obtain strictly (strongly) 2-distal theories with arbitrarily large finite arity. What about an example of infinite arity? Certainly, we can amalgamate the structures $\mathfrak{J}(k)$ into a multi-sorted structure with one sort for each $\mathfrak{J}(k)$, and the resulting structure would be strictly (strongly) 2-distal with infinite arity. However, this construction is rather contrived; we now encounter a much more natural one.
\subsubsection{A permutation structure of Cherlin and Lachlan}
The following structure appears in \cite[Section 1.6]{cherlinlachlan}, work of Cherlin and Lachlan. Let $M:=\{\{\{a, b\}, \{c, d\}\}: a,b,c,d\in\N\text{ distinct}\}$, and let $\Sym(\N)$ act on $M$ (and hence $M^k$ for each $k\in\N^+$) in the obvious way. For each $k\in\N^+$, there are finitely many orbits on $M^k$. There is a canonical structure on $M$ associated with this action, whose automorphism group is $\Sym(\N)$: for each $k\in\N^+$, we have a $k$-ary relation symbol for each orbit of $M^k$, interpreted as membership of the orbit.

The language $L$ is thus relational, with finitely many relation symbols of each arity. Let $T=\Th_L(M)$. As observed in \cite[Section 1.6]{cherlinlachlan}, $\ar(T)=\omega$. Indeed, fix $k\in\N^+$, and let $a_1, ..., a_{k+1}, b_1, ..., b_{k+1}\in\N$ be pairwise distinct. For each $i\in [k+1]$, let $m_i=\{\{a_i, a_{i+1}\}, \{b_i, b_{i+1}\}\}$, where $a_{k+2}:=a_1$ and $b_{k+2}:=b_1$; also, let $m'_{k+1}:=\{\{a_1, b_{k+1}\},\{b_1, a_{k+1}\}\}$. Then, the tuples $(m_1, ..., m_{k+1})$ and $(m_1, ..., m_k, m'_{k+1})$ witness that $\ar(T)>k$.
\begin{prop}\label{cherlinlachlan}
The theory $T$ is $\omega$-categorical, $\omega$-stable, and totally trivial --- hence strictly (strongly) 2-distal.    
\end{prop}
\begin{proof}
    Our argument is inspired by \cite[before Fact 2.6]{koponen}. Observe that $M$ is interpretable in $(\N,=)^{\text{eq}}$, so $T$ is $\omega$-categorical and $\omega$-stable. By \cite[Theorem 2.5.12]{anandbook}, such theories are 1-based. Suppose $T$ were not trivial. By \cite[Corollary 3.23]{depiro}, an infinite vector space is interpretable in every non-trivial, 1-based, $\omega$-categorical theory, and hence in $T$. Now, by \cite[after Theorem 2]{simonthomas}, $M$ is a reduct of a structure $M'$ that is homogeneous in a finite binary relational language, so $M'$ interprets an infinite vector space. This contradicts \cite[Theorem 1.1]{dugald}, which says that no infinite group is interpretable in a structure that is homogeneous in a finite relational language.

    Thus, $T$ is trivial. Since $\omega$-categorical $\omega$-stable theories have finite $U$-ranks \cite[Theorem 5.1]{omegacatomegastable}, by Proposition \ref{collapsefiniterank}, we have that $T$ is totally trivial. By Fact \ref{artemfrancis}, we conclude that $T_k$ is strictly (strongly) 2-distal.
\end{proof}
The proof above shows that every structure with $\omega$-categorical $\omega$-stable theory which has an expansion that is homogeneous for a finite relational language is totally trivial. This gives a different proof of Corollary \ref{Johnsonclassification}.
\subsection{Non-preservation under taking reducts}\label{subsecreducts}
A salient (and undesirable) feature of distality is that it is not preserved under reducts. Indeed, since the theory of the pure set is stable, it cannot be distal, and yet it is a reduct of every theory. It is, however, (strongly) 2-distal, so it does not give rise to an easy proof that (strong) $k$-distality is not preserved under taking reducts. Nonetheless, it is natural to expect this to be true, and we show that this is indeed the case.
\begin{prop}\label{reducts}
    For all $k\in\N^+$, $k$-distality and strong $k$-distality are not necessarily preserved under taking reducts.
\end{prop}
\begin{proof}
    It suffices to assume $k\geq 2$. In \cite{ampledividing}, Evans gives a theory $T_1'$ that is stable, 1-based, and trivial, and a reduct $T_1$ of $T'_1$ that is not trivial (in fact, he argues additionally that $T_1$ is not $k$-trivial for any $k\in\N^+$). Roughly speaking, $T_1'$ is the theory of a certain generic digraph, and $T_1$ is obtained by forgetting the orientation of the edges; moreover, $T_1$ is isomorphic to an \textit{ab initio} Hrushovski construction (see \cite{evans}).

    By \cite[Proposition 8]{goode}, stable, 1-based, and trivial theories are totally trivial, so in fact $T_1'$ is totally trivial. By Fact \ref{artemfrancis}, $T_1'$ is strongly 2-distal, and hence strongly $k$-distal for all $k\geq 2$. Now, $T_1$ is not $k$-trivial for any $k\in\N^+$, and hence not $k$-distal for any $k\geq 2$ by Fact \ref{walkerequiv}.
\end{proof}

On the topic of reducts, we observe a curious phenomenon concerning (strong) $k$-distality and reducts of $k$-ary theories. Recall Proposition \ref{k-aryisstronglyk-distal}, which says that every $k$-ary theory $T$ is strongly $k$-distal. The proof no longer goes through if $T$ is merely a \textit{reduct} of a $k$-ary theory. However, many examples of (strongly) $k$-distal theories are reducts of $k$-ary ones, and vice versa. The universal homogeneous kay-graph, which is (strongly) $k$-distal, was defined as a reduct of a $k$-ary theory; the Cherlin--Lachlan permutation structure, which is (strongly) 2-distal, is a reduct of a binary theory (see the proof of Proposition \ref{cherlinlachlan}); and even the Johnson graphs, which are (strongly) 2-distal, are reducts of binary theories (see \cite[Section 5.1]{bodirskybodor}). We do not imagine that every (strongly) $k$-distal theory is the reduct of a $k$-ary one, or vice versa, but we do not have counterexamples for either direction.
\subsection{Alternative proof of Fact \ref{artemfrancis}}\label{subsecalternativeproof}
We finish with our proof of Fact \ref{artemfrancis}, which says that $k$-total triviality and strong $(k+1)$-distality are equivalent among stable theories. The proof will go via a third condition that is, \textit{a priori}, a strengthening of strong $(k+1)$-distality. We restate Fact \ref{artemfrancis} to include this third condition.
\begin{prop}\label{equiv}
    Suppose $T$ is stable, and let $k\in\N^+$. The following are equivalent.
    \begin{enumerate}[(i)]
        \item The theory $T$ is $k$-totally trivial.
        \item The theory $T$ is strongly $(k+1)$-distal.
        \item Let $I$ be an infinite sequence of tuples from $\monster$, and let $b_1, ..., b_{k+1}$ be tuples from $\monster$. If $I$ is indiscernible over $b_{\neq m}$ for all $m\in [k+1]$, then $I$ is indiscernible over $b_1\cdots b_{k+1}$.
    \end{enumerate}
\end{prop}
The inclusion of condition (iii) is not necessary for our proof of Fact \ref{artemfrancis} (that (i) and (ii) are equivalent): indeed, we shall prove that (i) implies (iii) --- which immediately implies (ii) --- and (ii) implies (i). Rather, we isolate this condition to highlight that in the stable context, the hypothesis that $I+J$ is indiscernible over $b_1\cdots b_{k+1}$ can be removed from Definition \ref{defnstronglykdistal}. 

Coincidentally, Chernikov and Westhead also state and prove the equivalence between these three conditions in \cite[Proposition 4.16]{artemfrancis}; in their proof, they show that conditions (i) and (ii) are respectively equivalent to condition (iii). In their paper, they study condition (iii) as a property in its own right, calling it \textit{indiscernible $k$-triviality}. In particular, they prove that (ii) implies (iii) in \textit{extensible resilient} theories, a generalisation of stable theories \cite[Lemma 4.14]{artemfrancis}.

We collect some ingredients for our proof of Proposition \ref{equiv}, beginning with average types.
\begin{prop}\label{Av(I)}
    Suppose $T$ is stable. Let $I$ be an infinite indiscernible sequence. Define the global partial type
    \[\text{Av}(I):=\{\phi(x;b): b\in\monster, \models \phi(a;b)\text{ for cofinitely many }a\in I\}.\]
    Then $\text{Av}(I)$ is a complete type, and for all infinite subsets $I_0$ of $I$, $\text{Av}(I)$ is definable over $I_0$.
\end{prop}
\begin{proof}
    See \cite[Theorem 9.1.2]{tentziegler}.
\end{proof}
To simplify our proof, we will appeal to facts about strong types in stable theories; we thank Amador Martin-Pizarro for pointing these out to us. The first fact concerns the stationarity of strong types (Fact \ref{stpstationary}).
\begin{defn}
    For tuples $a,b\in\monster$ of length $l$ and a small set $D\subseteq \monster$, say that $a,b$ \textit{have the same strong type over $D$}, written $a\equiv^{\mathrm{stp}}_D b$, if for every finite $D$-definable equivalence relation $E$ between tuples of length $l$, $\models E(a,b)$.
\end{defn}
\begin{fact}\label{stpstationary}
    Suppose $T$ is stable. Let $a,b,c\in\monster$ be tuples and $D\subseteq \monster$ be a small set. If $a\indep_D c$, $b\indep_D c$, and $a\equiv^{\mathrm{stp}}_D b$, then $a\equiv_{Dc} b$.
\end{fact}
\begin{proof}
    See \cite[Proposition 9.19]{casanovas}.
\end{proof}
We will also need the following lemma.
\begin{lemma}\label{indiscstp}
    Let $D$ be a small set and $I$ be a $D$-indiscernible sequence. Then, for all $a,b\in I$, $a\equiv^{\mathrm{stp}}_D b$.
\end{lemma}
\begin{proof}
    Fix $a,b\in I$; without loss of generality, suppose $a\neq b$. Let $E$ be a finite $D$-definable equivalence relation between tuples of length $|a|$. Since $E$ has finitely many equivalence classes, there are distinct $a',b'\in I$ such that $\models E(a',b')$. But now, by $D$-indiscernibility, $\models E(a,b)$.
\end{proof}
For a global type $q$, a sequence $I=(a_n: n<\omega)$ of tuples from $\monster$, and $D\subseteq\monster$, write $I\models q^{(\omega)}\downharpoonright_D$ if $a_n\models q\downharpoonright_{Da_{<n}}$ for all $n<\omega$. Observe that if $q$ is $D$-invariant, then such $I$ is $D$-indiscernible.

We are now ready for our proof.
\begin{proof}[Proof of Proposition \ref{equiv}]
    We will prove the implications (iii) $\Rightarrow$ (ii) $\Rightarrow$ (i) $\Rightarrow$ (iii). Note that (iii) $\Rightarrow$ (ii) is immediate.

    Let us prove (ii) $\Rightarrow$ (i). Suppose $T$ is not $k$-totally trivial. By Lemma \ref{lemmamodel}, there are tuples $a, b_1, ..., b_{k+1}$ and a small model $M\models T$, such that $a\indep_M b_{\neq j}$ for all $j\in [k+1]$ but $a\not\indep_M b_1\cdots b_{k+1}$; write $b:=(b_1, ..., b_{k+1})$. By Lemma \ref{typesovermodels}, $\tp(a/M)$ is stationary; let $q$ be its unique global non-forking extension, which is $M$-invariant by Lemma \ref{stationaryinvariant}. By stationarity, we have $a\models q\downharpoonright_{b_{\neq j}M}$ for all $j\in [k+1]$. Take $I\models q^{(\omega)}\downharpoonright_{abM}$; then, $a+I$ is $b_{\neq j}M$-indiscernible for all $j\in [k+1]$. Fix any partition of $I$ into infinite subsequences $I_0$ and $I_1$. By total indiscernibility, $I_0+a+I_1$ is $b_{\neq j}M$-indiscernible for all $j\in [k+1]$, and so is $bM$-indiscernible as $T$ is strongly $(k+1)$-distal. But then $a\models q\downharpoonright_{bM}$, contradicting the fact that $a\not\indep_M b$.
    
    Finally, let us prove (i) $\Rightarrow$ (iii). Suppose (iii) fails. Then there is an infinite sequence $I$ without endpoints and tuples $b_1, ..., b_{k+1}$ such that, writing $b:=(b_1, ..., b_{k+1})$, $I$ is $b_{\neq i}$-indiscernible for all $i\in [k+1]$ but not $b$-indiscernible. Replacing $I$ with an infinite sequence of tuples from $I$ if necessary, we may assume that there are elements $a,a'\in I$ such that $a\not\equiv_b a'$.
    
    Write $I=I_0+I_1$ for infinite subsequences $I_0$ and $I_1$ such that $a,a'\in I_1$; note that $I_1$ is $I_0$-indiscernible, so $a\equiv^{\mathrm{stp}}_{I_0} a'$ by Lemma \ref{indiscstp}. By Proposition \ref{Av(I)}, $p:=\text{Av}(I)$ is definable over $I_0$, and hence $p$ does not fork over $I_0$ (by an easy exercise, or by \cite[Lemma 8.3.5]{tentziegler}). For all $i\in [k+1]$, since $I_1$ is $b_{\neq i}I_0$-indiscernible, we have that $a,a'\models p\downharpoonright_{b_{\neq i}I_0}$, so $a\indep_{I_0} b_{\neq i}$ and $a'\indep_{I_0} b_{\neq i}$. Since $T$ is $k$-totally trivial, $a\indep_{I_0} b$ and $a'\indep_{I_0} b$. Since $a\equiv^{\mathrm{stp}}_{I_0} a'$, we have by Fact \ref{stpstationary} that $a\equiv_{bI_0} a'$, which is a contradiction.
\end{proof}
\bibliography{bib}
\bibliographystyle{plainurl}
\end{document}